\documentclass[letterpaper, 10 pt, conference]{ieeeconf}  

\IEEEoverridecommandlockouts                              
\usepackage{versions}
\usepackage{dsfont}
\includeversion{techreport}
\excludeversion{cdc}

\usepackage{graphics} 
\usepackage{epsfig} 
\usepackage{subfiles}
\usepackage{newtxtext,newtxmath}
\usepackage{times} 
\usepackage{amsmath} 
\usepackage{amssymb}  
\usepackage{bbm}
\usepackage{tikz}
   \usepackage{amsthm}
\usepackage{url}

\newcommand{\ds}{\displaystyle}
\newcommand{\defn}{\stackrel{\triangle}{=}}

\newcommand{\tth}{{}^{\text{th}}}
\newcommand{\amax}{A_{\max}}
\newcommand{\smax}{S_{\max}}
\newcommand{\dto}{\downarrow}

\newcommand{\constant}{\beta}
\newcommand{\constantbar}{\overline{\constant}}

\newcommand{\bR}{\mathbb{R}}

\newcommand{\cS}{\mathcal{S}}
\newcommand{\cM}{\mathcal{M}}
\newcommand{\cC}{\mathcal{C}}
\newcommand{\cK}{\mathcal{K}}
\newcommand{\cF}{\mathcal{F}}

\newcommand{\cH}{\mathcal{H}}
\newcommand{\cT}{\mathcal{T}}
\newcommand{\cL}{\mathcal{L}}

\newcommand{\vq}{\boldsymbol{q}}
\newcommand{\va}{\boldsymbol{a}}
\newcommand{\vs}{\boldsymbol{s}}
\newcommand{\vu}{\boldsymbol{u}}

\newcommand{\vlambda}{\boldsymbol{\lambda}}
\newcommand{\vc}{\boldsymbol{c}}
\newcommand{\vnu}{\boldsymbol{\nu}}
\newcommand{\vmu}{\boldsymbol{\mu}}

\newcommand{\vx}{\boldsymbol{x}}
\newcommand{\vy}{\boldsymbol{y}}
\newcommand{\vw}{\boldsymbol{w}}

\newcommand{\ve}{\boldsymbol{e}}

\newcommand{\vone}{\boldsymbol{1}}
\newcommand{\vzero}{\boldsymbol{0}}

\newcommand{\qbar}{\overline{q}}
\newcommand{\vabar}{\overline{\va}}

\newcommand{\vsbar}{\overline{\vs}}

\newcommand{\vubar}{\overline{\vu}}
\newcommand{\ubar}{\overline{u}}
\newcommand{\Mbar}{\overline{M}}
\newcommand{\Bbar}{\overline{B}}

\newcommand{\vqpark}{\vq_{\parallel \cK}}
\newcommand{\vqperpk}{\vq_{\perp \cK}}
\newcommand{\vqbar}{\overline{\vq}}
\newcommand{\vqbarpark}{\vqbar_{\parallel \cK}}
\newcommand{\vqbarperpk}{\vqbar_{\perp \cK}}

\newcommand{\vqparh}{\vq_{\parallel \cH}}
\newcommand{\vqperph}{\vq_{\perp \cH}}
\newcommand{\vqbarparh}{\vqbar_{\parallel \cH}}
\newcommand{\vqbarperph}{\vqbar_{\perp \cH}}

\newcommand{\vabarparh}{\vabar_{\parallel \cH}}
\newcommand{\vsbarparh}{\vsbar_{\parallel \cH}}
\newcommand{\vubarparh}{\vubar_{\parallel \cH}}

\renewcommand{\pm}{{(m)}}
\newcommand{\pl}{{(\ell)}}
\newcommand{\pml}{{(m,\ell)}}
\newcommand{\peps}{{(\epsilon)}}

\newcommand{\vcl}{\vc^\pl}
\newcommand{\bl}{b^\pl}
\newcommand{\bml}{b^\pml}
\newcommand{\piml}{\pi^\pml}

\newcommand{\vcltilde}{\widetilde{\vc}^\pl}
\newcommand{\vqbartildel}{\widetilde{\vqbar}^\pl}
\newcommand{\vubartildel}{\widetilde{\vubar}^\pl}
\newcommand{\vqbarparhtilde}{\vqbartildel_{\parallel\cH}}
\newcommand{\vqbarperphtilde}{\vqbartildel_{\perp\cH}}

\newcommand{\ctildel}{\tilde{c}^\pl}

\newcommand{\ind}[1]{\mathds{1}_{\left\{#1 \right\}}}
\newcommand{\Prob}[1]{\mathbb{P}\left[#1\right]}
\newcommand{\E}[1]{\mathbb{E}\left[#1 \right]}

\newcommand{\Em}[1]{\mathbb{E}_m\left[#1 \right]}
\newcommand{\Var}[1]{\text{Var}\left[#1 \right]}

\makeatletter
\newcommand{\pushright}[1]{\ifmeasuring@#1\else\omit\hfill$\displaystyle#1$\fi\ignorespaces}

\def\ba#1\ea{\begin{align*}#1\end{align*}}
\def\ban#1\ean{\begin{align}#1\end{align}}

\theoremstyle{plain}
\newtheorem{theorem}{Theorem}
\newtheorem{lemma}{Lemma}
\newtheorem{proposition}{Proposition}

\allowdisplaybreaks
\title{\LARGE \bf
Logarithmic  Heavy Traffic Error Bounds in Generalized Switch and Load Balancing Systems
}
\newtheorem{claim}{Claim}
\newtheorem{corollary}{Corollary}

\author{ Daniela Hurtado-Lange, Sushil Mahavir Varma, Siva Theja Maguluri
\thanks{Daniela Hurtardo Lange and Sushil Mahavir Varma are students at Industrial and Systems Engineering, Georgia Institute of Technology
        {\tt\small d.hurtado@gatech.edu, sushil@gatech.edu}}%
\thanks{Siva Theja Maguluri is a faculty at Industrial and Systems Engineering, Georgia Institute of Technology
        {\tt\small siva.theja@gatech.edu}}%
}

\begin{document}

\maketitle
\thispagestyle{empty}
\pagestyle{empty}

\begin{abstract}
Motivated by application in wireless networks, cloud computing, data centers etc, 
Stochastic Processing Networks have been studied in the literature under various asymptotic regimes. In the heavy-traffic regime, the steady state mean queue length is proved to be $O(\frac{1}{\epsilon})$ where $\epsilon$ is the heavy-traffic parameter, that goes to zero in the limit. The focus of this paper is on obtaining queue length bounds on prelimit systems, thus establishing the rate of convergence to the heavy traffic. In particular, we study the generalized switch model operating under the MaxWeight algorithm, and we show that the mean queue length of the prelimit system is only $O\left(\log \left(\frac{1}{\epsilon}\right)\right)$ away from its heavy-traffic limit. We do this even when the so called complete resource pooling (CRP) condition is not satisfied. When the CRP condition is satisfied, in addition, we show that the MaxWeight algorithm is within $O\left(\log \left(\frac{1}{\epsilon}\right)\right)$ of the optimal. Finally, we obtain similar results in load balancing systems operating under the join the shortest queue routing algorithm. 


\end{abstract}


\section{Introduction}
Resource allocation and load balancing problems arise frequently in wide variety of applications such as wireless networks, data centers, ride hailing systems such as Uber and Lyft, routing and congestion control of traffic, manufacturing, telecommunications etc. It is typical to model these systems as Stochastic Processing Networks (SPNs) \cite{williams_survey_SPN}. Analyzing the delay and queue length behaviour of these systems in general is challenging and so, they are studied under various asymptotic regimes. Heavy traffic is a popular regime,  
where one studies the behavior of the system as the traffic intensity is increased to the maximum capacity. 
Even though it provides insights on the performance of the system asymptotically, a natural question is: `How well does the heavy traffic limiting behavior approximate the prelimit system?' Such a question can be answered by obtaining error bounds on the heavy-traffic approximation, as in \cite{atilla, MagSri_SSY16_Switch, Hurtado_gen-switch_temp}. 
In this paper, we obtain tight error bounds that grow logarithmically, as opposed to error bounds that grow polynomially in \cite{atilla, MagSri_SSY16_Switch,Hurtado_gen-switch_temp}.

Most of the work on heavy-traffic analysis is in systems that satisfy the so called Complete Resource Pooling (CRP) condition, which is satisfied when the system has a single bottleneck. In the heavy-traffic limit, the system then exhibits a State Space Collapse (SSC) onto a line, and behaves like a single-server queue. This makes the analysis tractable, and there are several different approaches to study such systems. Heavy traffic asymptotic performance of these systems was characterized using a diffusion limit approach in \cite{stolyar2004maxweight}, and using transform methods in \cite{Hurtado_transform_method_temp}. Lyapunov drift based arguments were used in \cite{atilla} to also obtain convergence rates, and it was shown that the steady-state mean of a linear combination of the queue lengths is $\frac{K_1}{\epsilon}+o\left(\frac{1}{\epsilon}\right)$ for some appropriately defined constant $K_1$, where $\epsilon$ is a parameter denoting the distance to the boundary of the capacity region.

In this paper we study a generalized switch model, which was first introduced in \cite{stolyar2004maxweight} to study several SPNs with control on the service process, such as input queued switches, ad hoc wireless networks, cloud computing, data centers etc. We consider the MaxWeight algorithm, and, using a tighter variant of the the drift argument in \cite{atilla,MagSri_SSY16_Switch,Hurtado_gen-switch_temp}, 
we show that MaxWeight is within $K_2\log \left(\frac{1}{\epsilon}\right)$ of the optimal policy (see Corollary \ref{corollary:gen-switch-CRP}). This is the first contribution of this paper.

We study a generalized switch without assuming that the CRP condition is satisfied, and we improve the bounds presented in \cite{Hurtado_gen-switch_temp} without adding any assumption. Specifically, we compute an upper bound of the form $\frac{K_1}{\epsilon}+K_2\log \left(\frac{1}{\epsilon}\right)$ for linear combinations of the queue lengths (see Theorem \ref{gs.thm:bounds}).
This establishes a logarithmically growing error bound with respect to the heavy traffic limit queue length behavior, $\frac{K_1}{\epsilon}$. This is the second contribution of this paper.

In addition to systems where the service is controlled, we look at load balancing systems, where the arrivals can be controlled. We consider the popular Join the Shortest Queue (JSQ) algorithm, which is known to satisfy the CRP condition, and exhibits one dimensional SSC \cite{atilla}. We show 
that the mean sum of the queue lengths is $\frac{K'_1}{\epsilon}+K'_2\log \left(\frac{1}{\epsilon}\right)$  (see Theorem \ref{thm:jsq}) which, in conjunction with the ULB showed in \cite{atilla}, establishes that JSQ is with in $K'_2\log \left(\frac{1}{\epsilon}\right)$ of the optimal. A similar result can be obtained for other algorithms such as power-of-$d$, which we don't present here due to lack of space. This is the third contribution of this paper.

A general resource allocation problem was studied in \cite{Mey_08}, under the CRP condition, and it was shown that $h$-MaxWeight algorithm, which is a variation of MaxWeight algorithm achieves logarithmic optimality. The function $h$ has to be found by solving a fluid control problem. 
In contrast, in this paper we show logarithmic optimality for vanilla MaxWeight algorithm in a generalized switch under CRP. 
More over, the results in this paper are also applicable to systems where CRP condition is not satisfied. 
 
\subsection{Notation} 

We denote the set of integers from 1 to $n$ by $[n]$. We denote the set of real numbers by $\mathbb{R}$, non negative real numbers by $\mathbb{R}_+$, integers by $\mathbb{Z}$ and non negative integers by $\mathbb{Z}_+$. All the vectors in the paper are boldfaced. 
The sets of $n$ dimensional vectors with real components and non negative real components are denoted by $\mathbb{R}$ and $\mathbb{R}_+$, respectively. We denote dot product between two vectors by $\langle \vx,\vy \rangle$ and Euclidean norm of a vector by $\|\vx\|$. We denote the $i\tth$ canonical vector by $\ve^{(i)}$, the vector of ones by $\vone$, and the vector of zeroes by $\vzero$. 
We denote transpose a matrix by $A^T$, and the Hadamard product between two matrices by $A \circ B$. The expectation and variance of a random variable $X$ are given by $\E{X}$ and $\Var{X}$, respectively, and the co-variance between two random variables $X$ and $Y$ by $Cov(X,Y)$. The probability of an event $E$ is denoted by $\Prob{E}$, and the indicator function of an event $E$ by $\ind{E}$. For a set $S$ we use $Int(S)$ and $Bo(S)$ to denote its relative interior and its boundary, respectively.

\section{Logarithmic  Error Bounds in Generalized Switch}
\subsection{Model}\label{sec:model}

In this section, we present the generalized switch model in detail. Consider $n$ queues 
operating in discrete time, with time indexed by $k \in \mathbb{Z}_+$. 
\begin{techreport}
A pictorial example is presented in Figure \ref{fig:model}.
\begin{figure}
    \centering
    \begin{tikzpicture}[scale=0.3]
    \draw[black, thick] (0,0) circle (1);
    \draw[black,thick] (-4,1) -- (-1,1) -- (-1,-1) -- (-4,-1);
    \draw[black, thick] (0,-2.5) circle (1);
    \draw[black,thick] (-4,-1.5) -- (-1,-1.5) -- (-1,-3.5) -- (-4,-3.5);
    \filldraw[color=black,fill=black,thick] (-2,-4.5) circle (0.15);
     \filldraw[color=black,fill=black,thick] (-2,-5.2) circle (0.15);
      \filldraw[color=black,fill=black,thick] (-2,-5.9) circle (0.15);
      \draw[black, thick] (0,-7.9) circle (1);
    \draw[black,thick] (-4,-6.9) -- (-1,-6.9) -- (-1,-8.9) -- (-4,-8.9);
    \draw[black,thick,->] (-7.5,0) -- (-4.5,0) node[anchor=south east] {\scriptsize $a_1(k)$};
    \draw[black,thick,->] (-7.5,-2.5) -- (-4.5,-2.5) node[anchor=south east] {\scriptsize $a_2(k)$};
    \draw[black,thick,->] (-7.5,-7.9) -- (-4.5,-7.9) node[anchor=south east] {\scriptsize $a_n(k)$};
    \draw[black,thick] (4,-3.95) -- (6.5,-1.95) -- (9,-3.95) -- (6.5,-5.95) -- cycle;
    \node at (6.5,-3.95){\scriptsize Scheduler};
    \draw[black,thick] (1,0) -- (4,-3.95);
    \draw[black,thick] (1,-2.5) -- (4,-3.95);
      \draw[black,thick] (1,-7.9) -- (4,-3.95);
      \draw[black,thick,->] (9,-3.95) -- (11,-3.95);
    \end{tikzpicture}
    \caption{Generalized switch model}
    \label{fig:model}
\end{figure}
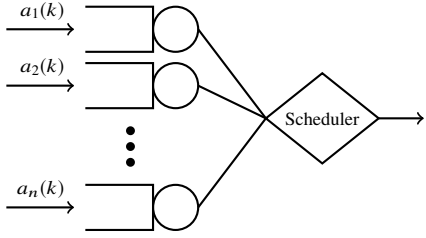
\end{techreport}
\subsubsection{Arrival Process}
We define a sequence of i.i.d. random variables $\{a_i(k): k \in \mathbb{Z}_+ \}$ for all $i \in [n]$, where $a_i(k)$ denotes the arrival to the $i\tth$ queue at time $k$. Denote the mean arrival rate vector as $\vlambda\defn \E{\va(1)}$ and the co-variance matrix of the random vector $\va(1)$ by $\Sigma_a$. Assume $a_i(1) \leq A_{\max}$ with probability 1 for all $i \in [n]$, where $\amax$ is a finite constant.
\subsubsection{Service Process}
Let $s_i(k)$ be the potential service that can be offered by server $i$ in time slot $k$. If there are not enough jobs to serve in the queue, there is unused service in that time slot, and we denote it by $u_i(k)$. Then, the actual number of served jobs in queue $i$ at time $k$ is $s_i(k)-u_i(k)$. We allow interference among the servers, which enforces them to satisfy a set of feasibility constraints in each time slot. The scheduler is allowed to pick any service rate vector which satisfies these constraints in each time slot. Additionally, the environment of the servers can affect the interference constraints, and we capture this by a sequence of i.i.d random variables $\{M(k): k \in \mathbb{Z}_+\}$, where $M(k)$ is the `channel state' at time slot $k$. We assume the state space of channel state is finite and denote it by $\cM$. In addition, let the pmf of $M(1)$ be $\psi_m\defn\Prob{M(1)=m}$ for all $m \in \cM$. Finally, denote the set of feasible service rates in channel state $m$ by $\cS^\pm$, and assume that $\cS^\pm$ contains the projection on the coordinate axes of its elements.
 We also assume the cardinality of $\cS^\pm$ for all $m \in \cM$ is finite. Thus, there exists a finite constant $\smax$ such that $s_i(1) \leq \smax$ with probability 1 for all $i\in[n]$.

\subsubsection{Queueing Process}
The following steps are followed in each time slot (in this order):
\begin{cdc}
(i) Observe the channel state; (ii) A scheduling problem is solved to determine which queues are served and the service rates according to the channel state; (iii) Arrivals occur; (iv) Jobs are processed according to the selected schedule.
\end{cdc}
\begin{techreport}
\begin{itemize}
    \item Observe the channel state.
    \item A scheduling problem is solved to determine which queues are served and the service rates according to the channel state.
    \item Arrivals occur in the system.
    \item Jobs are processed according to the selected schedule.
\end{itemize}
\end{techreport}
Then, the queue dynamics follows the following recursion. 
\begin{align}\label{eq:evolution}
    q_i(k+1)=q_i(k)\hspace{-2pt}+\hspace{-2pt}a_i(k)\hspace{-2pt}-\hspace{-2pt}s_i(k)\hspace{-2pt}+\hspace{-2pt}u_i(k) \ \hspace{-2pt}  \forall k \in \mathbb{Z}_+  \forall i \in [n]
\end{align}
If the unused service is positive, then the queue length at the start of the next time slot should be zero and vice versa. Thus, we have
\begin{align}
    q_i(k+1)u_i(k)=0 \quad \forall k \in \mathbb{Z}_+  \forall i \in [n]. \label{eq: qu}
\end{align}
The scheduling problem is solved using MaxWeight scheduling algorithm, which select the schedule with the maximum total weighted queue length. Mathematically, 
\begin{align}\label{gs.eq.MW}
	\vs(k)\in\arg\max_{\vx\in \cS^\pm}\langle \vq(k),\vx\rangle,
\end{align}
and ties are broken randomly. Observe that, unless there are ties, the potential service vector is deterministic after observing the channel state and the queue length vector.

\subsubsection{Capacity Region}
It is known that the capacity region of this system is $	\cC= \sum_{m\in \cM} \psi_m\,ConvexHull\left(\cS^\pm \right)$. Thus, it is a coordinate convex polytope \cite{atilla} and we write it as the intersection of finitely many half spaces, i.e., we write 
\begin{align}\label{gs.eq.cap.reg.pol}
	\cC= \left\{\vx\in\bR^n_+:\langle \vcl,\vx \rangle\leq\bl\;,\,\ell=1,\ldots,L  \right\}.
\end{align}
Without loss of generality, we assume $\vcl \geq \vzero$, $\|\vcl\|=1$ and $\bl>0$ for all $\ell \in [L]$. We also denote the $\ell\tth$ facet as $\cF^\pl \defn\left\{\vx\in\cC:\langle\vcl,\vx\rangle=\bl \right\}$. In addition, we denote the maximum $\vc^\pl$ weighted service rate by $\bml$. Mathematically, we have
\begin{align}\label{gs.eq.cl.weighted.def}
	\bml=\max_{\vx\in\cS^\pm}\langle \vcl,\vx\rangle \quad \forall \ell \in [L].
\end{align}
To capture the randomness in the service process due to the channel state, we define a sequence of i.i.d random variables (independent of queue lengths and arrival process) $\{B_\ell(k): k \in \mathbb{Z}_+\}$ with pmf given by $\Prob{B_\ell(1)=\bml}=\psi_m$. Let the correlation matrix of the vector $\left\{B_\ell(1)\right\}_{\ell \in [L]}$ be $\Sigma_B$.

\subsubsection{Heavy Traffic and State Space Collapse} Fix a vector $\vnu$ in the boundary of $\cC$ and let $\vlambda^\peps\defn (1-\epsilon)\vnu$. We analyze a sequence of generalized switches parametrized by $\epsilon$ and 
\begin{cdc}
we add a superscript $\peps$ to the variables to emphasize their dependence on $\epsilon$.
\end{cdc}
\begin{techreport}
denote the queue length, arrival process, service process and unused service for the $\epsilon\tth$ system by $\{\vq^{(\epsilon)}(k): k \in \mathbb{Z}_+\}$, $\{\va^{(\epsilon)}(k): k \in \mathbb{Z}_+\}$, $\{\vs^{(\epsilon)}(k): k \in \mathbb{Z}_+\}$ and $\{\vu^{(\epsilon)}(k): k \in \mathbb{Z}_+\}$ respectively.
\end{techreport}
The parametrization is such that $\E{\va^\peps(1)}=\vlambda^\peps$. Then, the heavy traffic regime is observed as $\epsilon\dto 0$.

Finally, we denote all the steady state vectors with a line on top of the variable. In particular, $\vqbar^\peps$ is the steady state queue length vector such that $\vq(k)$ converges in distribution to $\vqbar^\peps$ as $k \rightarrow \infty$ (which is well defined because the queue lengths process is positive recurrent for all $\epsilon\in(0,1)$). In addition, let $\vabar^{(\epsilon)}$, $\Mbar$, $\Bbar_\ell$  be the steady state random variable/vector with the same distribution as $\va^\peps(1)$, $M(1)$, $B_\ell(1)$ respectively. 
\begin{techreport}
We have $\E{\vabar^\peps}=\vlambda^\peps$ and denote the co-variance matrix of $\vabar^\peps$ by $\Sigma_a^\peps$. Also denote the steady state offered service by $\vsbar^\peps$ 
and the steady state unused service by $\vubar^\peps$.  
\end{techreport} 
Finally, let $\left(\vqbar^\peps\right)^+\defn \vqbar^\peps+\vabar^\peps-\vsbar^\peps+\vubar^\peps$ be the vector of queue lengths one time slot after $\vqbar^\peps$.

Define the cone $\cK$ spanned by the normal to the facets $\cF^\pl$ such that $\vnu \in \cF^\pl$. Let $P\defn \{\ell\in[L]: \vnu\in\cF^\pl\}$. It was proved in \cite{Hurtado_gen-switch_temp} that, under the parametrization described above, the state space collapses into $\cK$.   
So, we have
\begin{align}
& \cK=\left\{\vx\in\bR^n_+:\;\vx=\sum_{\ell\in P}\xi_\ell \vcl\;, \;\xi_\ell\geq 0 \;\forall \ell\in P \right\}.\label{gs.eq.coneK}
\end{align}
In addition, define $\cH$ as the affine hull of $\cK$. We also define $\tilde{P} \subset P$ as the maximal set of indices in $P$ such that $\left\{\vc^\pl : \ell \in \tilde{P}\right\}$ is a set of linearly independent vectors. Let $C\defn [\vc^\pl]_{\ell \in \tilde{P}}$ and observe $\cH$ is the column space of $C$.

\subsection{Logarithmic Error Bounds}
In this section, we present the main result of this paper. Specifically, we provide the error bounds of the expected value of linear combinations of the queue lengths as $\epsilon \dto 0$. After stating the result, we discuss two particular queueing systems in which it can be applied, and we present the proof at the end of the section. Then, in Section \ref{sec:gs.ssc} we prove SSC, which is an essential step in the proof of Theorem \ref{gs.thm:bounds}, and in Section \ref{sec:gs.proof} we prove the theorem.

\begin{theorem} \label{gs.thm:bounds}
Consider a set of generalized switches operating under MaxWeight scheduling policy and the heavy traffic parameter $\epsilon \in (0,1)$, as described in Section \ref{sec:model}. Then, there exists $\epsilon_0\in(0,1)$ such that for any $\epsilon<\epsilon_0$ and any vector $\vw \in \bigcap_{\ell \in P} \cF^\pl$, we have
\begin{align}\label{eq.gen.switch.thm.prelimit}
	&\left|\E{\langle\vqbar^\peps, \vw\rangle}- \dfrac{1}{2\epsilon} \vone^T\left(H\circ \Sigma_a^\peps\right)\vone  \right.\nonumber\\
	&\left.-\dfrac{1}{2\epsilon} \vone^T\left((C^TC)^{-1}\circ \Sigma_B \right)\vone \right|\leq \constant\log\left(\frac{1}{\epsilon}\right),
\end{align}
	where $H\defn C(C^T C)^{-1}C^T$ is the projection matrix into $\cH$ and $\constant$ is a constant independent of $\epsilon$. 
\end{theorem}

A similar result establishing the heavy traffic behavior in a generalized switch when CRP condition is not satisfied, was presented in \cite{Hurtado_gen-switch_temp}. 
The main difference 
is that while the result in \cite{Hurtado_gen-switch_temp} shows that the right hand side term in \eqref{eq.gen.switch.thm.prelimit} is $o\left(\frac{1}{\epsilon}\right)$, we obtain a tighter bound here.

Note that the result in Theorem \ref{gs.thm:bounds} presents a logarithmic error bound on the behavior of the MaxWeight algorithm, but does not characterize the optimality of MaxWeight algorithm. Such an optimality result can be obtained by proving a Universal Lower Bound (ULB) satisfied by any policy. Such ULBs were obtained in \cite[Proposition 1]{Hurtado_gen-switch_temp}, but they can be obtained only on particular linear combinations of queue lengths, which need not necessarily be the same as in Theorem \ref{gs.thm:bounds}. For a discussion about the conditions to ensure that both linear combinations coincide, see Remark 2 therein. Even when these conditions are satisfied, there are known examples such as a simple $2\times 2$ input queued switch \cite{2_2_switch_performance}, where the universal lower bound and the heavy-traffic limit of MaxWeight differ by a multiplicative constant (of less than 2). Thus, in general, MaxWeight need not be within an additive error of  $O\left(\log\left(\frac{1}{\epsilon}\right)\right)$ from the optimal policy. 



However, such a logarithmic optimality of MaxWeight can be obtained from Theorem \ref{gs.thm:bounds} when the CRP condition is satisfied and the state space collapses into a one dimensional subspace. In this case,  the heavy traffic limit is known to be the same as the ULB of the scaled expected linear combination of the queue length. In particular, fix $\ell\in[L]$ and assume $\vnu\in Int(\cF^\pl)$ so that the CRP condition is satisfied and SSC occurs into the line generated by $\vcl$. Then for any scheduling algorithm, we have
  \begin{align*}
        & \E{\langle\vqbar^\peps, \vcl\rangle} \geq \textit{ULB} \\ 
        &\defn \dfrac{1}{2\epsilon \bl}\left(\left(\vcl\right)^T \Sigma_a^\peps \vcl + \sigma_{B_\ell}^2 \right) - \dfrac{ (1-\epsilon)b_{\max}}{2},
    \end{align*}
where $b_{\max}\defn \max_{m\in\cM,\ell\in[L]}\{\bml\}$. By Theorem \ref{gs.thm:bounds}, we know that Max-Weight approaches the above lower bound as $\epsilon \dto 0$. Thus, Theorem \ref{gs.thm:bounds} 
establishes that  MaxWeight is within $O\left(\log\left(\frac{1}{\epsilon}\right)\right)$ of the optimal. 
We formally present this result in the next corollary.
\begin{corollary}\label{corollary:gen-switch-CRP}
    For the generalized switch operating under MaxWeight, as described in Theorem \ref{gs.thm:bounds}, fix $\ell\in[L]$ and assume $\vnu\in Int(\cF^\pl)$. 
    Then, for any $\epsilon<\epsilon_0$ we have
    \begin{align*}
        {\textit{ULB}} \leq \E{\langle\vqbar^\peps, \vcl\rangle} \leq {\textit{ULB}}+\tilde{\constant} \log\left(\frac{1}{\epsilon}\right),
    \end{align*}
    where $\tilde{\constant}$ is a constant independent of $\epsilon$ \begin{cdc}
    .
    \end{cdc}
    \begin{techreport}
    , which we compute in \eqref{eq:beta-tilde}. Hence, MaxWeight algorithm is heavy traffic optimal and the error bound to the optimal value is $\log\left(\frac{1}{\epsilon} \right)$.
    \end{techreport}
\end{corollary}
\begin{cdc}
We omit the proof of Corollary \ref{corollary:gen-switch-CRP} for brevity, but it is available in the technical report \cite{TechReport}.
\end{cdc}
\begin{techreport}
We present the proof of Corollary \ref{corollary:gen-switch-CRP} in Appendix \ref{app:proof-corollary-CRP}.
\end{techreport}
Such a logarithmic error bound in heavy traffic under CRP, was also obtained in \cite{Mey_08} in the context of a very general SPN model. The $h$-MaxWeight policy is a variant of the MaxWeight algorithm where the weight function $h$  is carefully chosen by solving a fluid control problem. In contrast, here we obtain a logarithmic error bound for the regular MaxWeight algorithm.

An immediate corollary of Theorem \ref{gs.thm:bounds} is to compute the bounds in the case of an input-queued switch. An input-queued switch is a discrete time model with $N^2$ queues that can be represented as $N \times N$ matrix. The $(i,j)\tth$ component of the matrix is the queue of packets at $i\tth$ input port, waiting to be processed at the $j\tth$ output port. Thus, rows are queues at input ports and columns are queues at output ports. All jobs take exactly one time slot to be processed and, in each time slot, at most one input/output pair can be served in each row and column. Then, the input-queued switch can be represented as a generalized switch where $n=N^2$, the channel state is fixed over time and the feasible service rate vectors are analogous to permutation matrices. Thus, the capacity region is given by $2N$ inequalities. Below we present the performance bound for this system under the assumption that all inequalities are tight at $\vnu$.
\begin{corollary}
    For the input queued switch defined above with independent arrivals, heavy traffic parameter $\epsilon \in (0,1)$, and $(\sigma_{a_i}^\peps)^2\defn \Sigma_{i,i}^\peps$, there exists a constant $\constantbar$ and $\epsilon_0\in(0,1)$ such that for all $\epsilon<\epsilon_0$
   	\begin{align*}
		\left|\E{\sum_{i=1}^{N^2} \qbar^\peps_i}- \left(1-\dfrac{1}{2N\epsilon}\right)\sum_{i=1}^{N^2}(\sigma_{a_i}^\peps)^2\right| \leq \constantbar \log\left(\frac{1}{\epsilon}\right)
	\end{align*}
\end{corollary}
The proof involves simplifying the left hand side of \eqref{eq.gen.switch.thm.prelimit} and is omitted as it is similar to \cite[Corollary 5]{Hurtado_gen-switch_temp}.

\subsection{State Space Collapse}\label{sec:gs.ssc}
\begin{cdc}
We start introducing notation. For a vector $\vx\in\bR^n_+$, let $\vx_{\parallel\cK}$ and $\vx_{\parallel\cH}$ be its projection on $\cK$ and $\cH$, respectively. Define $\vx_{\perp\cdot}\defn \vx-\vx_{\parallel\cdot}$ where $\cdot$ can be $\cK$ or $\cH$.
\end{cdc}

\begin{techreport}
We start introducing some notation. For each $\epsilon\in(0,1)$, let $\vqpark^\peps(k)$ and $\vqparh^\peps(k)$ be the projection of $\vq^\peps(k)$ on $\cK$ and $\cH$ respectively, and $\vqperpk^\peps(k)\defn \vq^\peps(k)-\vqpark^\peps(k)$, $\vqperph^\peps(k)\defn \vq^\peps(k)-\vqparh^\peps(k)$. Finally, we denote the steady state vectors by $\vqbarpark^\peps$, $\vqbarperpk^\peps$, $\vqbarparh^\peps$ and $\vqbarperph^\peps$ which are limit in distribution of $\left\{\vqpark^\peps(k):k\geq 1\right\}$, $\left\{\vqperpk^\peps(k):k\geq 1\right\}$, $\left\{\vqparh^\peps(k):k\geq 1\right\}$ and $\left\{\vqperph^\peps(k):k\geq 1\right\}$, respectively. The steady state vectors are well defined as the above Markov Chains are positive recurrent by the definition of projection and the fact that $\{\vq^\peps(k): k\geq 1\}$ is positive recurrent for all $\epsilon\in(0,1)$.
\end{techreport}

The SSC proved in \cite{Hurtado_gen-switch_temp} is that $\|\vqbar_{\perp \cK}\|$ has bounded moments, where the bounds do not depend on $\epsilon$. Here we explicitly compute a bound and, later, we use it to obtain the heavy traffic error bounds. 
\begin{proposition} \label{lemma: SSC}
For the generalized switch model operating under MaxWeight parametrized by $\epsilon\in(0,1)$ described in Section \ref{sec:model}, consider a vector $\vnu\in\textit{Bo}(\cC)$. 
Let $\delta>0$ be such that $\delta\leq b^\pl - \langle \vc^\pl,\vnu\rangle$ for all $\ell\in[L]\setminus P$ if $P\subsetneq [L]$ and $\delta=1$ if $P=[L]$. Let $\alpha\defn \max\{\amax,\smax\}$. If $\epsilon<\frac{\delta}{2\|\vnu\|}$, then for each $r=1,2,\ldots$ we have
\begin{align*}
    \E{\|\vqbar_{\perp \cH} \|^r}\leq{}& \E{\|\vqbar_{\perp \cK} \|^r} \\
    \leq{}& R_r\defn \left(\dfrac{8n\alpha^2}{\delta} \right)^r + (8\sqrt{n}\alpha)^r \left(\dfrac{8\sqrt{n}\alpha + \delta}{\delta} \right)^r r!
\end{align*}
\end{proposition}
\begin{cdc}
We omit the proof due to lack of space, but we present it in the technical report \cite{TechReport}.
\end{cdc}
\begin{techreport}
We present the proof in Appendix \ref{app:gen-switch-proof-SSC}.
\end{techreport}

\subsection{Proof of Theorem \ref{gs.thm:bounds}.}\label{sec:gs.proof}
Our proof is similar to the proof of \cite[Theorem 1]{Hurtado_gen-switch_temp}, so we omit some steps. Although, for clarity,
we will present a brief explanation wherever necessary. The main difference between our proof and the proof in \cite{Hurtado_gen-switch_temp} is that we compute tighter bounds for all traffic, and obtain logarithmic bounds on the heavy traffic error bounds. Before proving the theorem, we will restate the lemmas from \cite{Hurtado_gen-switch_temp} which are essential, for completeness.
\begin{lemma}\label{gs.lemma:bjl.bl}
	Let $\ell\in P$ and $m\in\cM$. Then, there exists $\vnu^\pm\in \cS^\pm$ such that $\bml=\langle \vcl,\vnu^\pm\rangle$. This implies that, for each $\ell\in P$, $\bl=\E{\Bbar_\ell}=\sum_{m\in\cM}\psi_m \bml$.
\end{lemma}

\begin{lemma}\label{gs.lemma:cl.s.bl}
	For each $m\in \cM$ and $\ell\in P$ define $\piml\defn \,\Prob{\left.\langle \vcl,\vsbar\rangle =\bml\,\right|\, \Mbar=m }$. 
	Then, $1-\piml$ is $O(\epsilon)$.
\end{lemma}

Now, we present the proof of the theorem.
\begin{proof}[Proof of Theorem \ref{gs.thm:bounds}]
We omit the dependence on $\epsilon$ of the variables for ease of exposition. 
\begin{cdc}
We set to zero the drift of $V_{\parallel \cH}(\vq)\defn\left\| \vq_{\parallel \cH}\right\|^2$, which is known to have finite expected value \cite{Hurtado_gen-switch_temp}. 
\end{cdc}
\begin{techreport}
We start by defining the Lyapunov function $V_{\parallel \cH}(\vq)\defn\left\| \vq_{\parallel \cH}\right\|^2$. To set the drift of this Lyapunov function to zero in the steady state, we first verify that $\E{\left\| \vq_{\parallel \cH}\right\|^2}<\infty$. This can be done using \cite[Lemma 2]{Hurtado_gen-switch_temp} and non expansive property of projection onto a convex set, but we omit the proof for brevity. 
\end{techreport}
We have

\begin{align}
0 ={}& \E{\left\|\vqbarparh^+ \right\|^2 - \left\|\vqbarparh \right\|^2} \nonumber \\
\overset{(*)}{=}&\underbrace{\E{\left\|\vabarparh - \vsbarparh\right\|^2}}_{\cT_2} +\underbrace{2\E{\langle \vqbarparh,\, \vabarparh-\vsbarparh \rangle}}_{-\cT_1} \nonumber\\
&-\underbrace{\E{\left\|\vubarparh \right\|^2}}_{\cT_3}+ \underbrace{2\E{\langle\vqbarparh^+,\, \vubarparh\rangle}}_{\cT_4} \label{gs.eq.drift.zero}
\end{align}
where $(*)$ follows by using the recursion \eqref{eq:evolution}, the definition of norm and inner product and then simplifying the terms. Thus, we have $\cT_1=\cT_2-\cT_3+\cT_4$. 
For this proof, we will borrow the bounds on $\cT_2$ and $\cT_3$ directly from \cite{Hurtado_gen-switch_temp} as these are not the bottleneck for the optimal error bounds. We restate them. There exist constants $\constant_1$ and $\constant_2$ such that
\begin{align}
& \left|\cT_2 - \vone^T\left(H\circ \Sigma_a^\peps\right)\vone - \vone^T \left((C^TC)^{-1}\circ \Sigma_B \vone 
\right) \right|
\leq \constant_1\epsilon \label{gs.eq.T2} \\
&\;\; \cT_3\leq \constant_2\epsilon \label{gs.eq.T3}
\end{align}
Now, we focus on the terms $\cT_1$ and $\cT_4$. We start with $\cT_1$.
\begin{align}
\cT_1=& 2\E{\langle\vqbarparh,\vsbarparh-\vabarparh \rangle} \nonumber \\
\stackrel{(*)}{=}& 2\epsilon\E{\langle\vqbarparh,\vnu\rangle} + 2\E{\langle\vqbarparh, \vsbar-\vnu\rangle}, \nonumber 
\end{align}
where $(*)$ follows by first using the orthogonality principle and then substituting $\E{\vabar}=(1-\epsilon)\vnu$ and observing that $\vabar$ is independent of $\vqbarparh$. It suffices to bound the second term. 
\begin{claim}\label{gs.claim.T1.partial}
	Consider the system described in Theorem \ref{gs.thm:bounds}. Then, there exist $\epsilon_0'>0$ and a finite constant $\constant_3$ such that
	\begin{align*}
	\left|\E{\langle \vqbarparh, \vsbar-\vnu\rangle} \right| 
	\leq \constant_3 \epsilon\log\left(\dfrac{1}{\epsilon}\right)\quad  \forall \epsilon<\epsilon_0'.
	\end{align*}
\end{claim}

For $\cT_4$ we have the following result.
\begin{claim} \label{claim: T4}
	Consider the system described in Theorem \ref{gs.thm:bounds}. Then, there exist $\epsilon_0''>0$ and a finite constant $\constant_4$ such that
\begin{align*}
    \cT_4\leq \constant_4 \epsilon\log\left(\frac{1}{\epsilon}\right) \quad \forall \epsilon < \epsilon_0''
\end{align*}
\end{claim}
The proof of both the claims are presented at the end of the section. Now, using \eqref{gs.eq.T2}, \eqref{gs.eq.T3}, and Claims \ref{gs.claim.T1.partial}, \ref{claim: T4}, we obtain that for any $\epsilon<\epsilon_0\defn \min\{\epsilon_0',\epsilon_0''\}$
\begin{align*}
\begin{aligned}
&\left|\E{\langle\vqbar^\peps, \vw\rangle}- \dfrac{1}{2\epsilon} \vone^T\left(H\circ \Sigma_a^\peps\right)\vone \right. \nonumber \\
&\left. - \frac{1}{2\epsilon}\vone^T\left((C^TC)^{-1}\circ \Sigma_B\right)\vone \right|\leq \constant \log\left(\frac{1}{\epsilon}\right),
\end{aligned}
\end{align*}
where $\beta$ depends on $\beta_1,\beta_2,\beta_3,\beta_4$.

To complete the proof for all $\vw\in\cap_{\ell\in P}\cF^\pl$ we follow the steps as in \cite[Proof of Theorem 1]{Hurtado_gen-switch_temp}, so we omit the details. 
\end{proof}

Now we prove the claims. The main idea is to use Hölder's inequality and Proposition \ref{lemma: SSC} with the right choice of the parameter $r$.
\begin{proof}[Proof of Claim \ref{gs.claim.T1.partial}]
Conditioning on the channel state, we get
\begin{align*}
&\E{\langle \vqbarparh, \vsbar-\vnu\rangle} 
\\
&\stackrel{(*)}{=} \sum_{m\in\cM}\psi_m \Em{ \langle \vqbarparh, \vsbar-\vnu^\pm\rangle \ind{\langle\vcl,\vsbar\rangle\neq \bml}},
\end{align*}
where $\vnu^\pm$ is defined as in Lemma \ref{gs.lemma:bjl.bl} and $(*)$ follows similarly as in \cite[Proof of Claim 1]{Hurtado_transform_method_temp}. 
It remains to show that $\Em{ \langle \vqbarparh, \vsbar-\vnu^\pm\rangle \ind{\langle\vcl,\vsbar\rangle\neq \bml}}$ is $O\left(\epsilon\log\left(\frac{1}{\epsilon}\right)\right)$.

Observe that $\vqbar=\vqbarparh+\vqbarperph = \vqbarpark+\vqbarperpk$, thus 

\begin{align}
& \Em{ \langle \vqbarparh, \vsbar-\vnu^\pm\rangle \ind{\langle\vcl,\vsbar\rangle\neq \bml}} \nonumber \\
& \begin{aligned}\label{gs.eq.claimT1.partial}
    &= \Em{\langle \vqbarpark,\vsbar-\vnu^\pm \rangle\ind{\langle\vcl,\vsbar\rangle\neq \bml}} \\
    &+ \Em{\langle \vqbarperpk-\vqbarperph ,\vsbar-\vnu^\pm\rangle\ind{\langle\vcl,\vsbar\rangle\neq \bml}}
\end{aligned}
\end{align}
Now, we show that each term in \eqref{gs.eq.claimT1.partial} is $O\left(\epsilon\log\left(\frac{1}{\epsilon}\right)\right)$. For the first term, we have 
$\Em{\langle \vqbarpark,\vsbar-\vnu^\pm \rangle\ind{\langle\vcl,\vsbar\rangle\neq \bml}} \leq 0$ 
by the definition of projection on the cone $\cK$ and by definition of $\vnu^\pm$ and $\bml$ in Lemma \ref{gs.lemma:bjl.bl}. Now, we have

\begin{align}
0\geq& \Em{\langle \vqbarpark,\vsbar-\vnu^\pm \rangle\ind{\langle\vcl,\vsbar\rangle\neq \bml}} \nonumber \\
\stackrel{(a)}{\geq}& - \E{\left\|\vqbarperpk \right\|^r}^{\frac{1}{r}} \Em{\left\| \vsbar-\vnu^\pm \right\|^p \ind{\langle\vcl,\vsbar\rangle\neq \bml}}^{\frac{1}{p}} \nonumber \\
\stackrel{(b)}{\geq}& -R_r^{\frac{1}{r}}\, \Em{\left\| \vsbar-\vnu^\pm \right\|^p \ind{\langle\vcl,\vsbar\rangle\neq \bml}}^{\frac{1}{p}} \nonumber
\end{align}
where $(a)$ holds by first using $\vqbarpark=\vqbar-\vqbarperpk$ and then bounding the term which includes $\vqbar$ using the definition of MaxWeight  in \eqref{gs.eq.MW}, and the fact that $\vnu^\pm\in\cS^\pm$. Finally, we bound the inner product between $\vqbar_{\perp \cK}$ and $\vsbar-\vnu^\pm$ using Hölder's inequality for some $p,r>1$ such that $\frac{1}{p}+\frac{1}{r}=1$. Inequality $(b)$ holds by SSC in Proposition \ref{lemma: SSC}. Now, by the definition of $R_r$, we have
\begin{align*}
    R_r^{\frac{1}{r}}&=\left(\left(\frac{8n\alpha^2}{\delta}\right)^r+\left(8\sqrt{n}\alpha\right)^r\left(\frac{8\sqrt{n}\alpha+\delta}{\delta}\right)^r r!\right)^{\frac{1}{r}} \\
    &\overset{(a)}{\leq} \constant_5 (r!)^{\frac{1}{r}} \overset{(b)}{\leq} \constant_5 e^{\frac{1}{r}-1} r^{1+\frac{1}{2r}} 
\end{align*}
where $\constant_5\defn \frac{8n \alpha^2}{\delta}+8\sqrt{n}\alpha\left(\frac{8\sqrt{n}\alpha+\delta}{\delta}\right)$; $(a)$ follows from Proposition \ref{lemma: SSC}; and $(b)$ follows from Stirling's approximation. Now, we will bound $\Em{\left\| \vsbar-\vnu^\pm \right\|^p \ind{\langle\vcl,\vsbar\rangle\neq \bml}}^{\frac{1}{p}}$ as follows.
\begin{align}
0\leq& \Em{\left\| \vsbar-\vnu^\pm \right\|^p \ind{\langle\vcl,\vsbar\rangle\neq \bml}}^{\frac{1}{p}} \nonumber \\
\stackrel{(a)}{=}& \Em{\left. \left\| \vsbar-\vnu^\pm \right\|^p \right| \langle\vcl,\vsbar\rangle\neq \bml}^{\frac{1}{p}}\left(1-\piml \right)^{\frac{1}{p}} \nonumber\\
\stackrel{(b)}{\leq}& n\left(\smax^p + V_{\max}^p\right)\left(1-\piml \right)^{\frac{1}{p}} \stackrel{(c)}{=} \constant_6\epsilon^{\frac{1}{p}} \label{eq.gs.claimT1.last},
\end{align}
where $(a)$ holds by definition of $\piml$ in Lemma \ref{gs.lemma:cl.s.bl}; $(b)$ holds with $V_{\max} = \max_{m\in \cM, i\in[n]}\nu^\pm_i$
; and  $(c)$ holds by Lemma \ref{gs.lemma:cl.s.bl} for $\constant_6\defn n\left(\smax^p + V_{\max}^p \right)\frac{\bml}{\gamma^\pm}$, where $\gamma^\pm\defn \min\left\{\bml-\langle\vcl,\vx\rangle: \langle \vcl,\vx\rangle<\bml ,\,\ell\in P,\, \vx\in\cS^\pm \right\}$.

Now, pick $r\defn\log\left(\frac{1}{\epsilon}\right)$ to get
\begin{align*}
    0\geq{}& \Em{\langle \vqbarpark,\vsbar-\vnu^\pm \rangle\ind{\langle\vcl,\vsbar\rangle\neq \bml}} \nonumber \\
    \geq{}& -\constant_5\constant_6e^{\frac{1}{r}-1}r^{1+\frac{1}{2r}}\epsilon^{\frac{1}{p}} \nonumber \\
    ={}& -\constant_5\constant_6e^{\frac{1}{\log\left(\frac{1}{\epsilon}\right)}-1}\log\left(\frac{1}{\epsilon}\right)^{1+\frac{1}{2\log\left(\frac{1}{\epsilon}\right)}}\epsilon^{-\frac{1}{\log\left(\frac{1}{\epsilon}\right)}}\epsilon \nonumber \\
    \overset{(*)}{\geq}{}& -2\constant_5\constant_6 \epsilon\log\left(\frac{1}{\epsilon}\right) \quad \forall \epsilon<\epsilon_0'
\end{align*}
where $\epsilon_0'$ is defined below, and $(*)$ follows as 
\begin{techreport}
\begin{align*}
    &\lim_{\epsilon \dto 0} e^{\frac{1}{\log\left(\frac{1}{\epsilon}\right)}-1}\log\left(\frac{1}{\epsilon}\right)^{\frac{1}{2\log\left(\frac{1}{\epsilon}\right)}}\epsilon^{-\frac{1}{\log\left(\frac{1}{\epsilon}\right)}}\nonumber \\
    &=\lim_{\epsilon \dto 0} e^{\frac{1}{\log\left(\frac{1}{\epsilon}\right)}-1} \lim_{\epsilon \dto 0} \log\left(\frac{1}{\epsilon}\right)^{\frac{1}{2\log\left(\frac{1}{\epsilon}\right)}}\lim_{\epsilon \dto 0} \epsilon^{-\frac{1}{\log\left(\frac{1}{\epsilon}\right)}} \nonumber \\
    &= \frac{1}{e} \times 1 \times e=1.
\end{align*}
\end{techreport}

\begin{cdc}
\begin{align*}
&\lim_{\epsilon \dto 0} e^{\frac{1}{\log\left(\frac{1}{\epsilon}\right)}-1}\log\left(\frac{1}{\epsilon}\right)^{\frac{1}{2\log\left(\frac{1}{\epsilon}\right)}}\epsilon^{-\frac{1}{\log\left(\frac{1}{\epsilon}\right)}} =1.
\end{align*}
\end{cdc}
By definition of limit, there exists $\epsilon_0'>0$ such that for all $\epsilon<\epsilon_0'$ we have
\begin{align*}
    e^{\frac{1}{\log\left(\frac{1}{\epsilon}\right)}-1}\log\left(\frac{1}{\epsilon}\right)^{\frac{1}{2\log\left(\frac{1}{\epsilon}\right)}}\epsilon^{-\frac{1}{\log\left(\frac{1}{\epsilon}\right)}} \leq 2.
\end{align*}
The proof that the second term \eqref{gs.eq.claimT1.partial} is $O\left(\epsilon\log\left(\frac{1}{\epsilon}\right)\right)$ follows similarly by linearity of dot product, Hölder's inequality with $r=\log\left(\frac{1}{\epsilon}\right)$ and \eqref{eq.gs.claimT1.last}. We omit the details for brevity.
\end{proof}
\begin{proof}[Proof of Claim \ref{claim: T4}]
We start with the following notation. For each $\ell\in P$, let $\cL_+^\pl\defn \left\{i\in[n]:\; c_i^\pl>0 \right\}$ 
\begin{techreport}
and define
\begin{align*}
\vcltilde=\left\{c_i^\pl\right\}_{i\in\cL_+^\pl} ,\, \vqbartildel=\left\{\qbar_i\right\}_{i\in\cL_+^\pl} \;\text{and}\; \vubartildel=\left\{\ubar_i\right\}_{i\in\cL_+^\pl}.
\end{align*}
\end{techreport}
\begin{cdc}
and for a vector $\vx\in\bR^n$ let $\tilde{\vx}\defn \left\{x_i\right\}_{i\in\cL_+^\pl}$. 
\end{cdc}
Then,
\begin{align*}
0\leq \left|\dfrac{\cT_4}{2} \right|={}& \left|\E{\langle\vqbarparh^+,\,\vubarparh\rangle} \right|
\stackrel{(a)}{=} \left|\E{- \langle \left(\vqbarperphtilde\right)^+,\,\vubartildel\rangle } \right| \\
\overset{(b)}{\leq}{}& \E{\left\|\left(\vqbarperphtilde\right)^+ \right\|^r}^{\frac{1}{r}}\E{\left\|\vubartildel\right\|^p}^{\frac{1}{p}}
\end{align*}
where $(a)$ follows using the definition of projection on the subspace to substitute $\vqbarparh^+=\sum_{\ell\in P}\langle\vcl,\vqbar^+\rangle\vcl$, then the key property \eqref{eq: qu}, and that $\left(\vqbartildel\right)^+=\left(\vqbarparhtilde\right)^+ +\left(\vqbarperphtilde\right)^+$. Then, $(b)$ holds by Hölder's inequality with $p,r>1$ such that $\frac{1}{r}+\frac{1}{p}=1$. Observe
\begin{align*}
\E{\left\|\left(\vqbarperphtilde\right)^+ \right\|^r}^{\frac{1}{r}}\leq \E{\left\|\vqbarperph^+\right\|^r}^{\frac{1}{r}}\leq T_r^{\frac{1}{r}} \leq \constant_1 e^{\frac{1}{r}-1} r^{1+\frac{1}{2r}}  
\end{align*}
and
\begin{align*}
0\leq{}& \E{\left\|\vubartildel\right\|^p}
\stackrel{(a)}{\leq} \sum_{\ell\in P}\sum_{i\in \cL_+^\pl} \dfrac{\ctildel_i}{\ctildel_i} \E{\widetilde{u}_i^p} \\
\stackrel{(b)}{\leq}{}& \dfrac{\smax^{p-1}}{\widetilde{c}_{\min}} \sum_{\ell\in P} \E{\langle\vcltilde,\,\vubartildel\rangle}
\stackrel{(c)}{\leq} \beta_7\epsilon
\end{align*}
where $(a)$ follows as the terms in the summation are all non-negative; $(b)$ holds by defining  $\ds\widetilde{c}_{\min}=\min_{\ell\in P, i\in[n]}\{\ctildel_i\}$ and by definition of dot product; and $(c)$ follows from \cite{Hurtado_gen-switch_temp} for a finite constant $\beta_7$. 
Now, pick $r \defn \log\left(\frac{1}{\epsilon}\right)$ to get
\begin{techreport}
\begin{align*}
    0 \leq \left|\dfrac{\cT_4}{2} \right| \leq{}& \constant_5\constant_7\dfrac{\smax^{1-\frac{1}{p}}}{\widetilde{c}_{\min}^{\frac{1}{p}}}|P|^{\frac{1}{p}} e^{\frac{1}{r}-1} r^{1+\frac{1}{2r}}  \epsilon^{\frac{1}{p}} \\
    ={}& \constant_5\constant_7\smax^{\frac{1}{\log\left(\frac{1}{\epsilon}\right)}}\left(\dfrac{|P|}{\widetilde{c}_{\min}}\right)^{1-\frac{1}{\log\left(\frac{1}{\epsilon}\right)}} e^{\frac{1}{\log\left(\frac{1}{\epsilon}\right)}-1}\times \nonumber\\
    &\log\left(\frac{1}{\epsilon}\right)^{1+\frac{1}{2\log\left(\frac{1}{\epsilon}\right)}}  \epsilon^{-\frac{1}{\log\left(\frac{1}{\epsilon}\right)}}\epsilon \nonumber \\
    \overset{(*)}{\leq}{}&2\constant_5\constant_7\frac{|P|}{\widetilde{c}_{\min}}\epsilon \log\left(\frac{1}{\epsilon}\right) \quad \forall \epsilon<\epsilon_0''
\end{align*}
where $(*)$ follows as 
\begin{align*}
    \lim_{\epsilon \dto 0} &\smax^{\frac{1}{\log\left(\frac{1}{\epsilon}\right)}}\left(\dfrac{|P|}{\widetilde{c}_{\min}}\right)^{1-\frac{1}{\log\left(\frac{1}{\epsilon}\right)}} e^{\frac{1}{\log\left(\frac{1}{\epsilon}\right)}-1}\log\left(\frac{1}{\epsilon}\right)^{\frac{1}{2\log\left(\frac{1}{\epsilon}\right)}}  \epsilon^{-\frac{1}{\log\left(\frac{1}{\epsilon}\right)}} \nonumber \\
    &=1 \times \dfrac{|P|}{\widetilde{c}_{\min}} \times \frac{1}{e} \times 1 \times e=\dfrac{|P|}{\widetilde{c}_{\min}}.
\end{align*}
Thus, there exists $\epsilon_0''>0$ such that for all $\epsilon<\epsilon_0''$ we have 
\begin{align*}
    \smax^{\frac{1}{\log\left(\frac{1}{\epsilon}\right)}}\left(\dfrac{|P|}{e\widetilde{c}_{\min}}\right)^{1-\frac{1}{\log\left(\frac{1}{\epsilon}\right)}} \log\left(\frac{1}{\epsilon}\right)^{\frac{1}{2\log\left(\frac{1}{\epsilon}\right)}}  \epsilon^{-\frac{1}{\log\left(\frac{1}{\epsilon}\right)}}\leq \dfrac{2|P|}{\widetilde{c}_{\min}}
\end{align*}
\end{techreport}
\begin{cdc}
\begin{align*}
    0 \leq \left|\dfrac{\cT_4}{2} \right| \leq{}& \constant_5\constant_7\dfrac{\smax^{1-\frac{1}{p}}}{\widetilde{c}_{\min}^{\frac{1}{p}}}|P|^{\frac{1}{p}} e^{\frac{1}{r}-1} r^{1+\frac{1}{2r}}  \epsilon^{\frac{1}{p}} \\
    \overset{(*)}{\leq}{}&2\constant_5\constant_7\frac{|P|}{\widetilde{c}_{\min}}\epsilon \log\left(\frac{1}{\epsilon}\right) \quad \forall \epsilon<\epsilon_0'',
\end{align*}
where $(*)$ follows similarly to the end of the proof of Claim \ref{gs.claim.T1.partial}. Details are presented in the technical report \cite{TechReport}.
\end{cdc}
\end{proof}
The key idea in obtaining a logarithmic error bound is in picking the right exponent $r$ in Hölder's inequality while bounding terms $\cT_1$ and $\cT_4$. We do this by minimizing the upper bound over $r$ (for a fixed $\epsilon$), which gives $r = \log\left(\frac{1}{\epsilon} \right)$. 

\section{Logarithmic Error Bounds in Load Balancing Systems}

While the generalized switch models many different SPNs where the control is on the service side, there are many systems of interest where the control is on the arrivals, such as the load balancing systems. 
The same technique that we introduced in the proof of Theorem \ref{gs.thm:bounds} can be used in used in such systems. In this section we present a result for a load balancing system operating under Join the Shortest Queue (JSQ) as an illustrative example. Similar results can be obtained for other routing algorithms such as power-of-$d$ choices. We first define the model.

\subsection{Load Balancing Model}\label{sec:jsq-model}

Consider an SPN with $n$ queues, each of them with a separate server. Arrivals occur in a single stream, and there is a dispatcher which routes them according to JSQ (i.e. to the server with the smallest number of jobs in line). After routing, jobs cannot commute lines. We model the system in discrete time, and we track the number of jobs in each queue. Then, the service policy is irrelevant.

Let $\{a(k): k\in\mathbb{Z}_+\}$ be the arrival process to the system, which is a sequence of i.i.d. random variables, and let $\va(k)$ be the vector of arrivals to the queues after routing at time $k$. Then, $a(k)=\sum_{i=1}^n a_i(k)$ by definition. Potential service is a sequence of i.i.d. random vectors, that we denote by $\{\vs(k):k\in\mathbb{Z}_+\}$, and it is independent of the arrival process. We assume there exist finite constants $\amax$ and $\smax$ such that $a(1)\leq \amax$ and $s_i(1)\leq \smax$ for all $i\in[n]$ with probability 1. The difference between potential and actual service is the unused service, and we denote $\vu(k)$ the unused service vector in time slot $k$. The dynamics of the queues occur according to \eqref{eq:evolution}, and \eqref{eq: qu} is satisfied for all $i\in[n]$.

Let $\vmu\defn \E{\vs(1)}$, $\sigma_{s_i}^2\defn \Var{s_i(1)}$ for each $i\in[n]$, and let $\mu_\Sigma \defn \sum_{i=1}^n \mu_i$. The capacity region of the load balancing model is $\cC=\left\{\lambda\in\bR_+: \lambda\leq \mu_\Sigma \right\}$ \cite{atilla}. To model heavy traffic we parametrize the arrival process by $\epsilon\in\left(0,\mu_\Sigma \right)$, letting $\lambda^\peps\defn \E{a^\peps(1)}=\mu_\Sigma-\epsilon$ and $\left(\sigma_a^\peps\right)^2\defn \Var{a^\peps(1)}$.

It is known that SSC occurs into the line where all queue lengths are equal in the load balancing system operating under JSQ. Specifically, \cite{atilla} proved that, denoting $\vqbar^\peps_{\parallel} \defn \left(\frac{1}{n}\sum_{i=1}^n \qbar^\peps_i\right)\vone$ and $\vqbar^\peps_\perp\defn \vqbar^\peps-\vqbar^\peps_{\parallel}$, we have $\E{\left\|\vqbar^\peps_\perp \right\|^r}$ is bounded for all $r\geq 1$.

\subsection{Logarithmic Error Bounds}\label{sec:jsq-rate}

The goal of this section is to prove the following result.

\begin{theorem}\label{thm:jsq}
    Consider a set of load balancing systems operating under JSQ and the heavy traffic parameter $\epsilon\in(0,\mu_\Sigma)$, as described above. Then, there exist a constant $\constant_{JSQ}$ and $\epsilon_0\in(0,\mu_\Sigma)$ such that for all $\epsilon<\epsilon_0$ 
    \begin{align*}
        \left|\E{\sum_{i=1}^n \qbar^\peps_i } -\dfrac{1}{2\epsilon}\left( \left(\sigma_a^\peps\right)^2 + \sum_{i=1}^n\sigma_{s_i}^2 \right) \right|\leq \constant_{JSQ}\log\left(\frac{1}{\epsilon} \right).
    \end{align*}
\end{theorem}

Similarly to the generalized switch, an essential step in the proof of Theorem \ref{thm:jsq} is to find an explicit upper bound for the moments of $\left\|\vqbar^\peps_\perp \right\|$. We present them in the next Proposition.

\begin{proposition}\label{prop:jsq-ssc}
    For the load balancing system operating under JSQ parametrized by $\epsilon\in(0,\mu_\Sigma)$ described in Section \ref{sec:jsq-model}, let $\mu_{\min}=\min_{i\in[n]}\mu_i$, $\delta\in(0,\mu_{\min})$ and $\alpha_{JSQ}\defn\max\{\amax,\smax\}$. Then, for any choice of $\epsilon\in(0,(\mu_{\min}-\delta)n)$, and all $r=1,2,\ldots$ we have
    \begin{align*}
        \E{\left\|\vqbar_\perp^\peps \right\|^r}\leq& R^{(JSQ)}_r\\
        \defn&{} \left(\dfrac{6n \alpha_{JSQ}^2}{\delta} \right)^r + \left(8\alpha_{JSQ}\sqrt{n
        } \right)^r \left(\dfrac{4\alpha_{JSQ} + \delta}{\delta} \right) r! 
    \end{align*}
\end{proposition}
\begin{cdc}
We omit the proof due to lack of space, but we present it in the technical report \cite{TechReport}.
\end{cdc}
\begin{techreport}
\begin{proof}[Proof of Proposition \ref{prop:jsq-ssc}]
    In \cite{atilla} it was proved that for all $k\geq 1$ we have
    \begin{align*}
        \left| \left\|\vq_{\perp\cK}(k+1) \right\| -\left\|\vq_{\perp\cK}(k) \right\|\right| \ind{\vq(k)=\vq}\leq 2\sqrt{n} \alpha_{JSQ}
    \end{align*}
    with probability 1, and that for all $\vq$ such that $\left\|\vq_{\perp\cK}\right\|\geq \frac{4n\alpha_{JSQ}^2}{\delta}$ we have
    \begin{align*}
        \E{\left. \left\|\vq_{\perp\cK}(k+1) \right\| -\left\|\vq_{\perp\cK}(k) \right\| \right| \vq(k)=\vq}\leq -\dfrac{\delta}{2}.
    \end{align*}
    Using these results in \cite[Lemma 3]{MagSri_SSY16_Switch} we obtain the result.    
\end{proof}
\end{techreport}

The proof of Theorem \ref{thm:jsq} follows from the computation of the upper bound in \cite{atilla}, similarly to the proof of Theorem \ref{gs.thm:bounds} so we omit it for brevity.


\section{Conclusions}
In this paper, we study the performance of generalized switch operating under MaxWeight both when the CRP condition is satisfied and when it is not. We show that MaxWeight is within $O\left(\log\left(\frac{1}{\epsilon}\right)\right)$ from its heavy traffic performance. When the CRP condition is satisfied, we show that it is within  $O\left(\log\left(\frac{1}{\epsilon}\right)\right)$ from the optimal policy. 

We also analyze the load balancing system operating under JSQ and prove that the rate of convergence of JSQ to the optimal heavy traffic performance under heavy traffic is $O\left(\log\left(\frac{1}{\epsilon}\right)\right)$. Similar results can be obtained for other routing algorithms.


\bibliographystyle{unsrt}
\bibliography{biblio-ok}

\begin{thebibliography}{1}

\bibitem{williams_survey_SPN}
Ruth Williams.
\newblock Stochastic processing networks.
\newblock {\em Annual Review of Statistics and Its Application}, 3:323--345,
  2016.

\bibitem{atilla}
A.~Eryilmaz and R.~Srikant.
\newblock Asymptotically tight steady-state queue length bounds implied by
  drift conditions.
\newblock {\em Queueing Systems}, 72(3-4):311--359, 2012.

\bibitem{MagSri_SSY16_Switch}
Siva~Theja Maguluri and R.~Srikant.
\newblock Heavy traffic queue length behavior in a switch under the
  {M}ax{W}eight algorithm.
\newblock {\em Stoch. Syst.}, 6(1):211--250, 2016.

\bibitem{Hurtado_gen-switch_temp}
Daniela Hurtado-Lange and Siva~Theja Maguluri.
\newblock Heavy-traffic analysis of queueing systems with no complete resource
  pooling.
\newblock 2019.
\newblock Technical Report \url{https://arxiv.org/pdf/1904.10096.pdf}.

\bibitem{stolyar2004maxweight}
A~Stolyar.
\newblock Max{W}eight scheduling in a generalized switch: State space collapse
  and workload minimization in heavy traffic.
\newblock {\em Annals of Applied Probability}, pages 1--53, 2004.

\bibitem{Hurtado_transform_method_temp}
Daniela Hurtado-Lange and Siva~Theja Maguluri.
\newblock Transform methods for heavy-traffic analysis.
\newblock 2019.
\newblock Technical Report \url{https://arxiv.org/abs/1811.05595}.

\bibitem{Mey_08}
S.P. Meyn.
\newblock Stability and asymptotic optimality of generalized {M}ax{W}eight
  policies.
\newblock {\em SIAM J. Control and Optimization}, 2008.
\newblock to appear.

\bibitem{2_2_switch_performance}
Y~Lu, ST~Maguluri, MS~Squillante, T~Suk, and X~Wu.
\newblock An optimal scheduling policy for the 2 x 2 input-queued switch with
  symmetric arrival rates.
\newblock {\em ACM SIGMETRICS Performance Evaluation Review}, 45(3):217--223,
  2018.

\end{thebibliography}


\begin{techreport}
\appendix
\subsection{Proof of Corollary \ref{corollary:gen-switch-CRP}}\label{app:proof-corollary-CRP}

\begin{proof}
    In this case we have $\cK=\left\{\xi \vcl: \xi\geq 0 \right\}$. In Theorem \ref{gs.thm:bounds} we can take $\vw=\bl\vcl$ and we have $H=\vcl (\vcl)^T$ because we assumed $\|\vcl\|=1$. Then, we obtain
    \begin{align*}
        & \E{\langle\vqbar^\peps, \vcl\rangle} \\
        &\leq \dfrac{1}{2\epsilon\bl} \left(\left(\vcl\right)^T \Sigma_a^\peps \vcl + \sigma_{B_\ell}^2 \right) + \constant \log\left(\dfrac{1}{\epsilon} \right),
    \end{align*}
    where $\sigma_{B_\ell}^2\defn (\Sigma_B)_{\ell,\ell}$ and $\beta$ is a constant that does not depend on $\epsilon$. Additionally, from \cite[Proposition 1]{Hurtado_gen-switch_temp} we obtain that, under any scheduling algorithm,
    \begin{align*}
        & \E{\langle\vqbar^\peps, \vcl\rangle} \\ 
        &\geq \dfrac{1}{2\epsilon\bl}\left(\left(\vcl\right)^T \Sigma_a^\peps \vcl + \sigma_{B_\ell}^2 \right) - \dfrac{b_{\max}(1-\epsilon)}{2} \\
        &= ULB,
    \end{align*}
    where $b_{\max}\defn \max_{m\in\cM,\ell\in[L]}\{\bml\}$. Then, putting both results together we obtain
    \begin{align*}
         & \E{\langle\vqbar^\peps, \vcl\rangle} \\
        &\leq ULB + \left(\dfrac{b_{\max}(1-\epsilon)}{2} + \constant \log\left(\dfrac{1}{\epsilon} \right) \right) \\
        & \stackrel{(*)}{\leq} ULB + \left(\dfrac{b_{\max}}{2}+\constant \right)\log\left(\dfrac{1}{\epsilon} \right),
    \end{align*}
    where $(*)$ holds because $1-x\leq \log\left(\frac{1}{x}\right)$ for all $x>0$. Therefore, defining
    \begin{align}\label{eq:beta-tilde}
        \tilde{\constant}\defn \dfrac{b_{\max}}{2}+\beta
    \end{align}
\end{proof}

\subsection{Proof of Proposition \ref{lemma: SSC}}\label{app:gen-switch-proof-SSC}

\begin{proof}[Proof of Proposition \ref{lemma: SSC}]
    The first inequality holds because $\cK\subset \cH$ and by definition of $\vqbar_{\perp\cK}$ and $\vqbar_{\perp \cH}$. Now we prove the second inequality. In \cite[Proposition 2]{Hurtado_gen-switch_temp}, it was proved that for all $k\geq 1$ we have
    \begin{align*}
        \left| \left\|\vq_{\perp\cK}(k+1) \right\| -\left\|\vq_{\perp\cK}(k) \right\|\right| \ind{\vq(k)=\vq}\leq 2\sqrt{n} \alpha
    \end{align*}
    with probability 1, and that for all $\vq$ such that $\left\|\vq_{\perp\cK}\right\|\geq \frac{4n\alpha}{\delta}$ we have
    \begin{align*}
        \E{\left. \left\|\vq_{\perp\cK}(k+1) \right\| -\left\|\vq_{\perp\cK}(k) \right\| \right| \vq(k)=\vq}\leq -\dfrac{\delta}{4}.
    \end{align*}
    Using these results in \cite[Lemma 3]{MagSri_SSY16_Switch} we obtain the result.
\end{proof}
\end{techreport}

\end{document}